\documentclass[12pt,a4paper,english]{amsart}

\usepackage[utf8]{inputenc}
\usepackage[T1]{fontenc}
\usepackage[english]{babel}
\usepackage{lmodern}
\usepackage{times}
\usepackage{graphicx}
\usepackage{latexsym}
\usepackage{dsfont}
\usepackage{amssymb,amsmath,amsfonts,amsthm}
\usepackage{array}
\usepackage[autostyle=true]{csquotes}
\usepackage{tikz}
\usetikzlibrary{matrix,arrows}
\newtheorem{Satz}{Satz}

\newtheorem{Lemma}[Satz]{Lemma}

\newtheorem{Theorem}[Satz]{Theorem}

\newtheorem*{Lemma*}{Lemma}

\theoremstyle{definition}

\newcommand{\R}{\mathbb{R}}

\newcommand{\N}{\mathbb{N}}
\newcommand{\T}{\mathbb{T}}

\newcommand{\cK}{\mathcal{K}}
\newcommand{\cB}{\mathcal{B}}

\newcommand{\cH}{\mathcal{H}}

\newcommand{\cP}{\mathcal{P}}
\newcommand{\cF}{\mathcal{F}}

\newcommand{\intd}{\mathrm{d}}

\newcommand{\1}{\mathds{1}}

\DeclareMathOperator{\GOp}{G}
\DeclareMathOperator{\AOp}{A}

\renewcommand{\S}{\mathbb S^{n-1}}
\newcommand{\Sn}{\mathbb S^{n-1}}

\newcommand{\MinTen}[3]
  {
    \Phi_{#1}^{#2,#3}
  }

\newcommand{\TenCM}[4]
  {
    \phi_{#1}^{#2,#3,#4}
  }

\newcommand{\IVol}[3]
  {
    V_{#1}
  }

\begin{document}

  \title{Integral geometric formulae for Minkowski tensors}
  \date{\today}

  \author{Daniel Hug \and Jan A. Weis}
  \address{Karlsruhe Institute of Technology (KIT), Department of Mathematics, D-76128 Karls\-ruhe, Germany}
  \email{daniel.hug@kit.edu}
  \address{Karlsruhe Institute of Technology (KIT), Department of Mathematics, D-76128 Karls\-ruhe, Germany}
  \email{jan.weis@kit.edu}

  \thanks{The authors were supported in part by DFG grants FOR 1548 and HU 1874/4-2}
  \subjclass[2010]{Primary: 52A20, 53C65; secondary: 52A22, 52A38, 28A75}
  \keywords{Kinematic formula, Crofton formula, tensor valuation, intrinsic volume, Minkowski tensor, integral geometry, convex body, polytope}

\begin{abstract}
    The Minkowski tensors are the natural tensor-valued generalizations of the intrinsic volumes of convex bodies.
    We prove two complete sets of integral geometric formulae, so called kinematic and Crofton formulae, for these Minkowski tensors.
    These formulae express the integral mean of the Minkowski tensors of the intersection of a given convex body with a second geometric object (another convex body in the kinematic case and an affine subspace in the Crofton case) which is uniformly moved by a proper rigid motion, in terms of linear combinations of the Minkowski tensors of the given geometric objects.
\end{abstract}

  \maketitle

  \section{Introduction}

    One of the origins of classical integral geometry is a series of lectures given by Wilhelm Blaschke in Germany, Bulgaria, and Rumania in the 1930s (collected in \cite{Blaschke55}). He initiated investigations on questions in the field of convex and differential geometry, which arise from problems in classical geometric probability, but are nevertheless of geometric interest, independent of their probabilistic applications. Integral geometry basically deals with averaging certain geometric functionals with respect to invariant measures. In particular, intersection formulae are integral geometric key results, which describe integral mean values with respect to invariant measures of specific geometric quantities evaluated at the intersection of a moving geometric object and a fixed geometric object. For a classical approach to this topic see \cite{Santalo71},  recent developments are described in the monographs \cite[Chap. 5]{SchnWeil08} and \cite[Chap. 4.4]{Schneider14}.

    The two best known and most important classical intersection formulae are the \textit{principal kinematic formula} and the \textit{classical Crofton formula}, which treat the intersection of a convex body in Euclidean space with another geometric object (in the kinematic formula this is another convex body, in the Crofton formula this is an affine subspace) which is uniformly moved by a proper rigid motion. More precisely, denoting the space of convex bodies (nonempty, compact, convex sets) in Euclidean space $\R^{n}$ by $\cK^{n}$, the principal kinematic formula (see \cite[(4.52)]{Schneider14}) for the {\em intrinsic volumes} $V_i$, $i\in \{0,\ldots,n\}$, states that, for any two convex bodies $K, K' \in \cK^{n}$ and $j \in \{ 0, \ldots, n \}$,
    \begin{align} \label{Form_Princ_KF}
      \int_{\GOp_n} V_{j} (K \cap g K') \, \mu( \intd g) = \sum_{k = j}^{n} \alpha_{n, j, k} V_{k}(K) V_{n - k + j}(K'),
    \end{align}
    where $\GOp_{n}$ denotes the group of proper rigid motions of $\R^{n}$, $\mu$ is the motion invariant Haar measure on $\GOp_{n}$, normalized in the usual way (see \cite[p. 586]{SchnWeil08}), and the constant
    \begin{align*}
      \alpha_{n, j, k} :=
        \frac{\Gamma \left( \frac{k + 1} {2} \right) \Gamma \left( \frac{n - k + j + 1} {2} \right)} {\Gamma \left( \frac{j + 1} {2} \right) \Gamma \left( \frac{n + 1} {2} \right)}
    \end{align*}
    is expressed in terms of specific values of the Gamma function.
    Furthermore, the classical Crofton formula (see \cite[(4.59)]{Schneider14}) for the intrinsic volumes states that, for a convex body $K \in \cK^{n}$, $k \in \{ 0, \ldots, n \}$, and $j \in \{ 0, \ldots, k \}$,
    \begin{equation} \label{Form_Croft_Classic}
      \int_{\AOp(n, k)} V_{j}
      (K \cap E) \, \mu_{k} (\intd E) = \alpha_{n, j, k} V_{n - k + j}
      (K),
    \end{equation}
    where $\AOp(n, k)$ is the affine Grassmannian of $k$-flats in $\R^{n}$, $\mu_{k}$ denotes the motion invariant Haar measure on $\AOp(n, k)$, normalized as in \cite[p. 588]{SchnWeil08}, and $\alpha_{n,j,k}$ is the same constant as in the principal kinematic formula \eqref{Form_Princ_KF}.

    The functionals $V_{j}: \cK^{n} \rightarrow \R$, $j \in \{ 0, \ldots, n \}$, appearing in \eqref{Form_Princ_KF} and \eqref{Form_Croft_Classic}, are the {intrinsic volumes}, which are the uniquely determined coefficients of the monomials in the \textit{Steiner formula}
    \begin{equation} \label{Form_Steiner}
      \mathcal{H}^n (K + \varepsilon B^{n}) = \sum_{j = 0}^{n} \kappa_{n - j} V_{j} (K) \varepsilon^{n - j},\qquad \varepsilon \geq 0,
    \end{equation}
    which holds for all convex bodies $K \in \cK^{n}$. As usual in this context, $+$ denotes the Minkowski addition in $\R^{n}$,  $B^{n}$ is the Euclidean unit ball in $\R^{n}$ of $n$-dimensional volume $\kappa_{n}$, and $\mathcal{H}^n$ is the $n$-dimensional Hausdorff measure.
    Properties of the intrinsic volumes such as continuity, isometry invariance, additivity (valuation property) and homogeneity are derived from corresponding properties of the volume functional. It is well known that the intrinsic volumes can be uniquely extended by additivity to finite
    unions of convex bodies (polyconvex sets). As an immediate consequence, all integral geometric results in this contribution  hold more generally for polyconvex sets.
    A key result for the intrinsic volumes is \textit{Hadwiger's characterization theorem}, which states that $V_{0}, \ldots, V_{n}$ form a basis of the vector space of continuous and isometry invariant real-valued valuations on $\cK^{n}$ (see \cite[Theorem~6.4.14]{Schneider14}).
    This theorem can be used to derive not only \eqref{Form_Princ_KF} and \eqref{Form_Croft_Classic}, but also Hadwiger's general integral geometric theorem (see \cite[Theorem~5.1.2]{SchnWeil08}).

    In the early 1970s, Hadwiger \& Schneider (see \cite{HadSchn71}) and Schneider (see \cite{Schneider72, Schneider72a}) introduced a vector-valued generalization of the intrinsic volumes (which are also called \textit{quermassintegrals}), the so called \textit{quermassvectors} or \textit{curvature centroids}.
    In fact, they proved a characterization theorem, similar to the aforementioned theorem by Hadwiger, for vector-valued valuations.
    Furthermore, they established kinematic and Crofton formulae for the quermassvectors by an application of their characterization theorem.
    More recently,  McMullen (1997) made one step further and initiated a study of tensor-valued generalizations of the intrinsic volumes (see \cite{McMullen97}).
    In the same contribution, he suggested as the \enquote{ultimate aim} to find a description of the vector space of continuous and isometry covariant tensor-valued (of a fixed rank) valuations on $\cK^{n}$, in the spirit of Hadwiger's paradigmatic characterization theorem in the scalar-valued case.
    Only two years later, Alesker showed that this vector space is spanned by the so-called \textit{Minkowski tensors} (see \cite{Alesker99a, Alesker99b}), which are natural tensor-valued generalization of the intrinsic volumes.
    In contrast to the intrinsic volumes in the scalar case and to the quermassvectors in the rank one case, the Minkowski tensors of rank greater than one satisfy nontrivial linear relationships, which were already discovered by McMullen (see \cite[Theorem 5.3]{McMullen97}). In particular, they do not form a basis of the vector space they span. Thereafter, Hug, Schneider \& Schuster proved that these linear relationships are essentially the only ones, and thus they explicitly determined  the dimension of the corresponding vector space (see \cite{HugSchnSchu08b}).

    Alesker's characterization result for the Minkowski tensors is naturally connected to  integral geometric formulae for the Minkowski tensors.  In fact, Alesker's characterization theorem already implies that the kinematic and Crofton integrals for Minkowski tensors can again be expressed in terms of Minkowski tensors.
    However, the existence of linear dependences makes it substantially harder (if possible at all) to determine these intersection formulae explicitly, only on the basis of a global characterization result (as it can be done conveniently in the scalar and in the rank one case). By a substantially different approach,
    Hug, Schneider \& Schuster obtained a complete set of Crofton formulae for the Minkowski tensors (see \cite[Theorems~2.1--2.6]{HugSchnSchu08a}).
    An annoying  drawback of the explicit form of these formulae was the complicated (though explicit) form of the coefficients required to express the Crofton integral as a linear combination of Minkowski tensors.
    By now there exist several further approaches (by different authors) to integral formulae for Minkowski tensors (for instance, see \cite{BernHug15} for translation invariant Minkowski tensors, \cite{KousKideHug15, HugWeis16a} for intrinsically defined Minkowski tensors, and the literature cited there).

    The aim of the present article is to state and prove for the first time a complete set of kinematic and Crofton formulae for the Minkowski tensors. Our approach is  based on the recent kinematic and Crofton formulae for tensorial curvature measures, derived in \cite{HugWeis16b, HugWeis16c}, in combination with a relation going back to McMullen (see \cite{McMullen97}), which had been used before.
    Surprisingly, the formulae we will derive via globalization of the tensorial curvature measures (which can be viewed as local versions of the Minkowski tensors) now involve concise and structurally simple coefficients also in the case of the Crofton integrals.

    During the last  two decades, Minkowski tensors have been discussed and applied in the literature related to physics and other natural sciences. In these contexts, disordered spatial structures are explored, the properties of which are determined by their intrinsic geometry. Minkowski tensors turned out to be a perfect and versatile tool for analyzing and quantitatively measuring  such structures (see the surveys \cite{Mecke00, SchrEtAl10, SchrEtAl11, KideVede17} for an extensive overview and the PhD thesis \cite{Klatt16} for various detailed investigations).
    An exemplary  list of applications includes nuclear physics \cite{SchuEtAl15}, granular matter \cite{KapfEtAl12, XiaEtAl14, SchaEtAl15, KuhnSunWang15}, density functional theory \cite{WittMareMeck14}, physics of complex plasmas \cite{BoebRaet16}, and physics of materials science \cite{SaadEtAl12}.
    These and many other applications rest upon various characterization and classification theorems for tensor valuations, uniqueness and reconstruction results \cite{HoerKous16, KousKide16, Kous16, Kous16a}, accompanied by numerical algorithms \cite{SchrEtAl10, SchrEtAl11, HugKideSvan16, ChriKide16}, stereological estimation procedures \cite{KousKideHug15, KousEtAl16}, and integral geometric formulae \cite{HugSchnSchu08a, HugWeis16a, SvanJens16}.

    The paper is structured as follows.
    In Section \ref{secPre}, we fix our notation and define the Minkowski tensors.
    In Sections \ref{Sec_KF_MT} and \ref{Sec_CF_MT}, we state the kinematic and the Crofton formulae. For the proof our integral geometric formulae, we recall in Section \ref{secprep} (an iteration of) a connexion between Minkowski tensors and certain total generalized tensorial curvature measures from McMullen \cite[p.~269]{McMullen97}. The main proofs are provided in Section \ref{Sec_Proofs}. In each case, we first deal with the case of translation invariant Minkowski tensors.
    As the arguments heavily rely on the corresponding formulae for tensorial curvature measures, we recall the definition of these measure-valued valuations and state the required results in a form adjusted to the current application, in an appendix.

  \section{Preliminaries}\label{secPre}

    We work in the $n$-dimensional Euclidean space $\R^{n}$, equipped with its usual topology generated
    by the standard scalar product $\langle \cdot\,, \cdot \rangle$ and the corresponding Euclidean norm
    $\| \cdot \|$.
    For a topological space $X$, we denote the Borel $\sigma$-algebra on $X$ by $\cB(X)$.

    We recall from the introduction that $\GOp_{n}$ denotes the group of proper rigid motions of $\R^{n}$, and $\mu$ is the motion invariant Haar measure on $\GOp_{n}$, normalized in the usual way (see \cite[p. 586]{SchnWeil08}).
    For $k \in \{0, \ldots, n\}$, we denote the affine Grassmannian of $k$-dimensional  affine subspaces of $\R^{n}$ by $\AOp(n, k)$.
    We write  $\mu_{k}$ for the motion invariant Haar measure on $\AOp(n, k)$, normalized as in \cite[p. 588]{SchnWeil08}.
    The directional space of an affine $k$-flat $E \in \AOp(n, k)$ is denoted by $E^{0} \in \GOp(n, k)$ and its orthogonal complement by $E^{\perp} \in \GOp(n, n - k)$.

    The algebra of symmetric tensors over $\R^n$ is denoted by $\T$ (the underlying $\R^n$ will be clear from the context), the vector space of symmetric tensors of rank $p\in\N_0$ is denoted by $\T^p$ with $\T^0=\R$.
    The symmetric tensor product of two symmetric tensors $T_1, T_2$ over $\R^{n}$ is denoted by $T_1T_2$, and the $p$-fold symmetric tensor product of a symmetric tensor $T$ by $T^{p}$, $p\in\N_0$, where $T^0:=1$.
    Identifying $\R^{n}$ with its dual space via its scalar product, we consider a symmetric
    tensor of rank $p$ as a symmetric $p$-linear map from $(\R^{n})^{p}$ to $\R$.
    A special tensor is the \emph{metric tensor} $Q \in \T^{2}$, defined by $Q(x, y) := \langle x, y \rangle$ for $x, y \in \R^{n}$.
    For an affine $k$-flat $E\subset\R^n $, $k \in \{0,\ldots, n\}$, the metric tensor $Q(E)$ associated with $E$ is defined by $Q(E)(x, y) := \langle x_{E^{0}}, y_{E^{0}} \rangle$ for $x, y \in \R^{n}$, where $x_{E^{0}}$ denotes the orthogonal projection of $x$ to $E^0$.

    In order to define the Minkowski tensors and to explain how they can be epxressed in terms of the support measures,
    we start with some preparations.
    For a convex body $K \in \cK^{n}$ and $x \in \R^{n}$, we denote the metric projection of $x$ onto $K$ by $p(K, x)$ and define $u(K, x) := (x - p(K, x)) / \| x - p(K, x) \|$ for $x \in \R^{n} \setminus K$.
    For $\varepsilon > 0$ and a Borel set $\eta \subset \Sigma^{n}:=\R^{n} \times \Sn$,
    \begin{equation*}
      M_{\varepsilon}(K, \eta) := \left\{ x \in \left( K + \varepsilon B^{n} \right)
      \setminus K \colon \left( p(K, x), u(K, x) \right) \in \eta \right\}
    \end{equation*}
    is a local parallel set of $K$ which satisfies the \emph{local Steiner formula}
    \begin{equation} \label{14-Form_Steiner_loc}
      \mathcal{H}^n (M_{\varepsilon}(K, \eta)) = \sum_{j = 0}^{n - 1} \kappa_{n - j} \Lambda_{j} (K, \eta)
        \varepsilon^{n - j}, \qquad \varepsilon \geq 0.
    \end{equation}
    This relation determines the \emph{support measures} $\Lambda_{0} (K, \cdot), \ldots, \Lambda_{n - 1} (K, \cdot)$ of $K$, which are finite Borel measures on $\cB (\Sigma^{n})$.
    Obviously, a comparison of \eqref{14-Form_Steiner_loc} and the Steiner formula yields $V_{j}(K) = \Lambda_{j} (K, \Sigma^{n})$.
    For further information see \cite[Chap.~4.2]{Schneider14}.

    The \textit{Minkowski tensors} can be seen as the tensor-valued generalizations of the intrinsic volumes, obtained by extending the relation between the intrinsic volumes and the support measures to tensorial functions and measures. Hence, for a convex body $K \in \cK^{n}$ and $j, r, s \in \N_{0}$, we define the Minkowski tensor
    \begin{equation*}
      \MinTen{j}{r}{s} (K) := \frac{1}{r! s!} \frac{\omega_{n - j}}{\omega_{n - j + s}} \int _{\R^{n} \times \Sn} x^r u^s \, \Lambda_j(K, \intd (x, u)),
    \end{equation*}
    for $j \in \{ 0, \ldots, n - 1 \}$, where $\omega_{n}$ denotes the $(n - 1)$-dimensional volume of $\Sn$, and by
    \begin{equation*}
      \MinTen{n}{r}{0} (K): =\Phi_n^{r}(K):= \frac{1}{r!} \int _{K} x^r \, \cH^{n}(\intd x).
    \end{equation*}
    For the sake of convenience, we extend these definitions by $\MinTen{j}{r}{s} := 0$ for $j \notin \{ 0, \ldots, n \}$ or $r \notin \N_{0}$ or $s \notin \N_{0}$  or for  $j = n$ and $s \neq 0$.

    For a polytope, there is an alternative representation of the support measures, which also yields an explicit description for the Minkowski tensors.
    Let $\cP^n\subset\cK^n$ denote the space of convex polytopes in $\R^n$.
    For a polytope $P \in \cP^{n}$ and $j \in \{ 0, \ldots, n \}$, we
    denote the set of $j$-dimensional faces of $P$ by $\cF_{j}(P)$ and
    the normal cone of $P$ at a face $F \in \cF_{j}(P)$ by $N(P,F)$.
    Then, the $j$th support measure $\Lambda_{j} (P, \cdot)$ of $P$ is explicitly given by
    \begin{equation*}
      \Lambda_{j} (P, \eta) = \frac {1} {\omega_{n - j}}
      \sum_{F \in \cF_{j}(P)} \int _{F} \int _{N(P,F) \cap \Sn}
      \1_\eta(x,u)
      \, \cH^{n - j - 1} (\intd u) \, \cH^{j} (\intd x)
    \end{equation*}
    for $\eta \in \cB(\Sigma^{n})$ and $j \in \{ 0, \ldots, n - 1 \}$,
		where $\cH^{j}$ denotes the $j$-dimensional Hausdorff measure.
    In the same spirit, the Minkowski tensor of $P$ is given by
    \begin{align} \label{Form_MT_Pol}
      \MinTen{j}{r}{s} (P) =  \frac{1}{r! s!} \frac{1}{\omega_{n - j + s}}
			\sum_{F \in \cF_{j}(P)} \int _{F} x^r \, \cH^{j}(\intd x)
			\int _{N(P,F) \cap \Sn} u^{s} \, \cH^{n - j - 1} (\intd u),
    \end{align}
    for $j \in\{0, \ldots, n - 1\}$.%

  \section{Kinematic formulae for Minkowski tensors} \label{Sec_KF_MT}

    In this section, we state the complete set of kinematic formulae for the Minkowski tensors of convex bodies. In other words, for $K, K' \in \cK^{n}$ we express the integral mean value
    \begin{align*}
      \int_{\GOp_{n}} \MinTen{j}{r}{s} (K \cap g K') \, \mu( \intd g)
    \end{align*}
    in terms of certain Minkowski tensors of $K$ and $K'$. In fact, only some of the Minkowski tensors of $K$ are required and, in particular, only scalar Minkowski tensors (that is, intrinsic volumes) of $K'$.

    We proceed in two steps. First, we state and prove the formulae for the translation invariant Minkowski tensors, then we turn to the formulae for general Minkowski tensors. It is natural to consider the translation invariant Minkowski tensors separately, since the involved Minkowski tensors are linearly independent in this case and the result has a simpler form.
    The proof is basically an application of the kinematic formulae for the  tensorial curvature measures (obtained in \cite{HugWeis16b}), combined with conclusion \eqref{Form_McM_MT_gtCM} after Lemma \ref{Lem_McM_MT_gtCM}.

  \subsection{Translation invariant Minkowski tensors}

      As explained before, we start by stating the kinematic formula for translation invariant Minkowski tensors $\MinTen{j}{0}{s}$, $j, s \in \N_{0}$ with $j \leq n$, where $s = 0$ if $j = n$.
      Here, the proof and the representation of the kinematic integrals are simpler than in the general case, since  several coefficients of the occurring Minkowski tensors can be combined in a suitable way.

      \begin{Theorem} \label{Thm_KF_tiMT}
        Let $K, K' \in \cK^n$ and $j, s \in \N_{0}$ with $j \leq n$, where $s = 0$ if $j = n$. Then
        \begin{align*}
          & \int_{\GOp_{n}} \MinTen{j}{0}{s} (K \cap g K') \, \mu( \intd g) = \sum_{k = j}^{n} \sum_{m = 0}^{\lfloor \frac s 2 \rfloor} e_{n, j, k}^{s, m, 0} \, Q^{m} \MinTen{k}{0}{s - 2m} (K) \, \IVol{n - k + j}{0}{0} (K'),
        \end{align*}
        where, for $m = 0, \ldots, \lfloor \frac s 2 \rfloor - 1$,
        \begin{align*}
           e_{n, j, k}^{s, m, 0} : = \frac{1}{(4\pi)^{m} m!} \frac {\Gamma(\frac {k} 2)} {\Gamma(\frac {j} 2)} \frac {\Gamma(\frac {j + s} 2 - m)} {\Gamma(\frac {k + s} 2)} \frac {\Gamma(\frac {k - j} 2 + m)} {\Gamma(\frac{k - j}{2})} \alpha_{n, j, k},
        \end{align*}
        with $\alpha_{n, j, k}$ as in \eqref{Form_Princ_KF}, and
        \begin{align*}
          e_{n, j, k}^{s, \lfloor \frac s 2 \rfloor, 0} :=
            \frac{\frac{k} 2 + \lfloor \frac s 2 \rfloor}{(4\pi)^{\lfloor \frac s 2 \rfloor} \lfloor \frac s 2 \rfloor!}
            \frac {\Gamma(\frac {k} 2)} {\Gamma(\frac {j} 2 + 1)}
            \frac {\Gamma(\frac {j + s} 2 - \lfloor \frac s 2 \rfloor + 1)} {\Gamma(\frac {k + s} 2 + 1)}
            \frac {\Gamma(\frac {k - j} 2 + \lfloor \frac s 2 \rfloor)} {\Gamma(\frac{k - j}{2})} \alpha_{n, j, k},
        \end{align*}
        except for the case  $k=j=0$ where $e_{n, 0, 0}^{s, m, 0}:=\1\{m=0\}$ for $m=0,\ldots,\lfloor\frac{s}{2}\rfloor$.
      \end{Theorem}

      In Theorem \ref{Thm_KF_tiMT} one could define the coefficient $e_{n, j, k}^{s, \lfloor \frac s 2 \rfloor, 0}$ as for general $m$, since the definitions coincide for even $s$, and for odd $s$ the difference  does not have any effect, since $\MinTen{k}{0}{1} \equiv 0$.
      However, for subsequent use we have already defined them here in the appropriate way.

      For $k = j$, we note that the coefficient in Theorem \ref{Thm_KF_tiMT} is given by
      $ e_{n, j, j}^{s, m, 0}  = \1\{ m = 0 \}$.

      For $k = n$, we note that the Minkowski tensors $\MinTen{n}{0}{s - 2m}$ vanish if $m \neq \frac s 2$. In the case of $m = \frac s 2$ (and hence $s$ even), the corresponding coefficient is given by
      \begin{align*}
        e_{n, j, n}^{s, \frac s 2, 0} & = \frac{1}{(2 \sqrt\pi)^{s} \frac{s}{2}!} \frac {\Gamma(\frac {n} 2)} {\Gamma(\frac {n + s} 2)} \frac {\Gamma(\frac {n - j + s} 2)} {\Gamma(\frac{n - j}{2})}.
      \end{align*}

      If $j = 0$, then the kinematic integral equals zero, if $s$ is odd (the same is true for the sum on the right side). For even $s$, the only non-vanishing coefficients  for $k \in \{ 1, \ldots, n \}$ on the right side of the kinematic formula are $e_{n, 0, k}^{s, \frac s 2, 0}$, as in that case the ratio ${\Gamma(\frac {j + s} 2 - m)} / {\Gamma(\frac {j} 2)}$ is read as $\1 \{ m = \frac s 2 \}$, since for $m = \frac s 2$ the two Gamma functions cancel and for $m < \frac s 2$ the term vanishes.

  \subsection{General Minkowski tensors}

      In this section, we state the kinematic formula for general Minkowski tensors.
      The representation of the kinematic integral will be more involved as compared to the translation invariant case treated in Theorem \ref{Thm_KF_tiMT}. This is due to the fact that  Lemma \ref{Lem_McM_MT_gtCM} contributes Minkowski tensors to the formula which did not occur before and cannot  be combined with the other Minkowski tensors.
      Nevertheless, the explicit representation of the integral geometric formula is still surprisingly simple (in particular if compared to the Crofton formulae in \cite{HugSchnSchu08a}, which can be applied to derive kinematic formulae).

      \begin{Theorem} \label{Thm_KF_MT}
        Let $K, K' \in \cK^n$ and $j, r, s \in \N_{0}$ with $j \leq n$, where $s = 0$ if $j = n$. Then
        \begin{align*}
          \int_{\GOp_{n}} \MinTen{j}{r}{s} (K \cap g K') \, \mu( \intd g) & = \sum_{k = j}^{n} \sum_{p = 0}^{r} \sum_{m = 0}^{\lfloor \frac s 2 \rfloor} e_{n, j, k}^{s, m, p} \, Q^{m} \MinTen{k + p}{r - p}{s - 2m + p}(K)  \IVol{n - k + j}{0}{0}(K'),
        \end{align*}
        where the coefficients $e_{n, j, k}^{s, m, 0}$ are defined as in Theorem \ref{Thm_KF_tiMT}.
        For  $p = 1, \ldots, r$ and $m = 0, \ldots, \lfloor \frac s 2 \rfloor - 1$ the coefficients are
        \begin{align*}
          e_{n, j, k}^{s, m, p} & : = \frac{m \tfrac{k - p}{k} - \tfrac{s + p}{2} \tfrac{k - j}{k}}{(4\pi)^{m} m!}
          \frac {\Gamma(\frac {k} 2 + 1)} {\Gamma(\frac {j} 2 + 1)}
          \frac {\Gamma(\frac {j + s} 2 - m)} {\Gamma(\frac {k + s} 2 + 1)}
          \frac {\Gamma(\frac {k - j} 2 + m)} {\Gamma(\frac{k - j}{2})} \alpha_{n, j, k},
        \end{align*}
         with $\alpha_{n, j, k}$ as in \eqref{Form_Princ_KF}, and
        \begin{align*}
          e_{n, j, k}^{s, \lfloor \frac s 2 \rfloor, p} :=
            \frac{1}{(4\pi)^{\lfloor \frac s 2 \rfloor} \lfloor \frac s 2 - 1 \rfloor!}
            \frac {\Gamma(\frac {k} 2)} {\Gamma(\frac {j} 2 + 1)}
            \frac {\Gamma(\frac {j + s} 2 - \lfloor \frac s 2 \rfloor + 1)} {\Gamma(\frac {k + s} 2 + 1)}
            \frac {\Gamma(\frac {k - j} 2 + \lfloor \frac s 2 \rfloor)} {\Gamma(\frac{k - j}{2})} \alpha_{n, j, k},
        \end{align*}
        except for the case $j=k=0$ where $e_{n, j, k}^{s, m, p}:=0$ for $p\ge 1$ and $m=0,\ldots,\lfloor \frac s 2 \rfloor$.
      \end{Theorem}

      Note that $e_{n, j, k}^{s, \lfloor \frac s 2 \rfloor, p}=0$ for $s=0$, since $((-1)!)^{-1}=\Gamma(0)^{-1}=0$.
      In Theorem \ref{Thm_KF_MT},  we also have $e_{n, j, j}^{s, m, p} = 0$ (for $j=0$ this holds by definition). If $p > 0$ and $k = n$, then the Minkowski tensors $\MinTen{k + p}{r - p}{s - 2m + p}$ vanish. For further simplifications of the coefficients, see the remark after Theorem \ref{Thm_KF_tiMT}.

  \section{Crofton formulae for Minkowski tensors} \label{Sec_CF_MT}

    In this section, we state the complete set of Crofton formulae for Minkowski tensors, which can be derived from the corresponding formulae for tensorial curvature measures in~\cite{HugWeis16c}. That is, for $K \in \cK^{n}$, we explicitly express integrals of the form
    \begin{align*}
      \int_{\AOp(n, k)} \MinTen{j}{r}{s} (K \cap E) \, \mu_k( \intd E)
    \end{align*}
    as a linear combination of Minkowski tensors of $K$ (multiplied with suitable powers of the metric tensor). As for the kinematic formulae, we only need a some of the Minkowski tensors.

   We start with the case $j = k$. For the sake of completeness, we mention this formula here, although it has already been derived in \cite{HugSchnSchu08a} by a completely different approach.

    \begin{Theorem} \label{Thm_CF_MT_j=k}
      Let $K \in \cK^n$ and $k, r, s, l \in \N_{0}$ with $k \leq n$, where $s = 0$ if $k = n$. Then,
      \begin{align*}
        \int_{\AOp(n, k)} \MinTen{k}{r}{s} (K \cap E) \, \mu_k( \intd E) = \1 \{ s \text{ even} \} \frac{1}{(4\pi)^{\frac s 2} \frac{s}{2}!} \frac{\Gamma(\frac{n}{2}) \Gamma(\frac{n - k + s}{2})}{\Gamma(\frac{n + s}{2}) \Gamma(\frac{n - k}{2})} \, Q^{\frac s 2} \MinTen{n}{r}{0}(K).
      \end{align*}
    \end{Theorem}

    It is not necessary to state a proof of Theorem \ref{Thm_CF_MT_j=k}.
    In fact, the formula is an immediate consequence of Theorem \ref{Thm_CF_tCM_j=k} (see also \cite[Theorem~1]{HugWeis16c}), which is obtained by simply setting $\beta = \R^{n}$.
    If $k = n$, then the Minkowski tensor on the left side of the formula vanishes if $s \not = 0$, and so does the factor $\Gamma(\frac{n - k + s}{2}) / \Gamma(\frac{n - k}{2})$ on the right side.

  \subsection{Translation invariant Minkowski tensors}

    We proceed with the Crofton formulae in the case of $j < k$, and start with the translation invariant Minkowski tensors.

    \begin{Theorem} \label{Thm_CF_tiMT_j<k}
      Let $K \in \cK^{n}$ and $j, k, r, s \in \N_{0}$ with $j < k \leq n$, where $s = 0$ if $j = n$. Then,
      \begin{align*}
        \int_{\AOp(n, k)} \MinTen{j}{0}{s} (K \cap E) \, \mu_k( \intd E) = \sum_{m = 0}^{\lfloor \frac s 2 \rfloor} e_{n, j, n - k + j}^{s, m, 0} \, Q^{m} \MinTen{n - k + j}{0}{s - 2m} (K),
      \end{align*}
      where the coefficients $e_{n, j, n - k + j}^{s, m, 0}$ are defined as in Theorem \ref{Thm_KF_tiMT}.
    \end{Theorem}

    For $j = 0$, we have $\Gamma(\frac {j + s} 2 - m) / \Gamma(\frac {j} 2) = \1 \{ m = \frac s 2 \}$. Thus in this case, the only remaining summand on the right-hand side of the Crofton formula in Theorem \ref{Thm_CF_tiMT_j<k} is $e_{n, 0, n-k}^{s, \frac s 2,0} \, Q^{\frac s 2} \MinTen{n - k + j}{0}{0} (K)$, if $s$ is even (otherwise the integral on the left-hand side vanishes).
    We note that Theorem \ref{Thm_CF_tiMT_j<k} coincides with Theorem 3 in \cite{BernHug15}, which was derived by a completely different algebraic approach.

  \subsection{General Minkowski tensors}

    Finally, we state the Crofton formula for general Minkowski tensors.
    Similar to the kinematic formulae, we conclude from Lemma \ref{Lem_McM_MT_gtCM} that the representation of the Crofton integrals involves more Minkowski tensors than in the translation invariant case.

    \begin{Theorem} \label{Thm_CF_MT_j<k}
      Let $K \in \cK^{n}$ and $j, k, r, s \in \N_{0}$ with $j < k \leq n$. Then,
      \begin{align*}
        \int_{\AOp(n, k)} \MinTen{j}{r}{s} (K \cap E) \, \mu_k( \intd E) & = \sum_{p = 0}^{r}  \sum_{m = 0}^{\lfloor \frac s 2 \rfloor} e_{n, j, n - k + j}^{s, m, p} \, Q^{m} \MinTen{n - k + j + p}{r - p}{s - 2m + p}(K),
      \end{align*}
      where the coefficients $e_{n, j, n - k + j}^{s, m, p}$ are defined as in Theorem \ref{Thm_KF_tiMT} and  \ref{Thm_KF_MT}.
    \end{Theorem}

    In Theorem \ref{Thm_CF_MT_j<k}, for $j = k - 1$, the only summand remaining in the summation with respect to $p$ is the one for $p = 0$, as the Minkowski tensors, which occur for $p > 0$, vanish.

\section{Preparations}\label{secprep}

    In this section, we state a lemma proved by McMullen (see \cite[p.~269]{McMullen97}), which provides a relation between Minkowski tensors and
    certain total generalized tensorial curvature measures (further details on these can be found in \cite{HugWeis16b}).
    Following the presentation in \cite{McMullen97}, for a polytope $P \in \cP^{n}$ with face $F \in \cF_{k}(P)$, $k \leq n$, we define
    \begin{align*}
      \Upsilon_{r}(F) := \frac{1}{r!} \int_{F} x^{r} \, \cH^{k}(\intd x),
    \end{align*}
    for $r \in \N_{0}$, and
    \begin{align*}
      \Theta_{s}(P, F) :=\frac{1}{s!} \frac{1}{\omega_{n - k + s}} \int_{N(P, F) \cap \S} u^{s} \, \cH^{n - k - 1} (\intd u),
    \end{align*}
    for $s \in \N_{0}$. If $r \notin \N_{0}$, then we set $\Upsilon_{r}(F) := 0$, and if $s \notin \N_{0}$, then $\Theta_{s}(P, F) := 0$.
     Then we obtain
    \begin{align*}
      \MinTen{k}{r}{s}(P) = \sum_{F \in \cF_{k}(P)} \Upsilon_{r}(F) \Theta_{s}(P, F),
    \end{align*}
    using the just defined functionals. Furthermore, for $k\in\{0,\ldots,n-1\}$ and integers $r,s$, we obtain
    \begin{align*}
      \TenCM{k}{r}{s - 2}{1}(P, \R^{n}) = \tfrac{2\pi}{k} \sum_{F \in \cF_{k}(P)} Q(F) \Upsilon_{r}(F) \Theta_{s - 2}(P, F),
    \end{align*}
    where $\TenCM{k}{r}{s - 2}{1}(P, \R^{n})$ are the total generalized tensorial curvature measures
    of $P$ (for definitions and further remarks, see Appendix \ref{Sec_Appendix}).

    Next, we can state McMullen's lemma, which has been proved in a different way by Ralph Schuster (see  {\cite[Lemma~2.3.5]{Schuster04}}).
    \begin{Lemma}[{\cite[p.~269]{McMullen97}}] \label{Lem_McM_MT_gtCM}
      Let $P \in \cP^{n}$ be a polytope, let $r, s $ be integers and $k \in \{ 0, \ldots, n \}$. Then
      \begin{align*}
        2 \pi s \, \MinTen{k}{r}{s}(P) & = \sum_{F \in \cF_{k}(P)} Q(F^{\perp}) \Upsilon_{r}(F) \Theta_{s - 2}(P, F) \\
        & \qquad  + \sum_{G \in \cF_{k + 1}(P)} Q(G) \Upsilon_{r - 1}(G) \Theta_{s - 1}(P, G).
      \end{align*}
    \end{Lemma}

    For $r = 0$, the second sum on the right side of the formula in Lemma \ref{Lem_McM_MT_gtCM} vanishes. If also $s = 1$, then the lemma simply states
    that $\MinTen{k}{0}{1} \equiv 0$.

    Moreover, we obtain a representation of $\TenCM{k}{r}{s - 2}{1}(P, \R^{n})$ by Minkowski tensors, as Lemma~\ref{Lem_McM_MT_gtCM} yields
    \begin{align*}
      \tfrac{k}{2\pi} \TenCM{k}{r}{s - 2}{1}(P, \R^{n}) & = \sum_{F \in \cF_{k}(P)} Q(F) \Upsilon_{r}(F) \Theta_{s - 2}(P, F) \\
      & = Q \MinTen{k}{r}{s - 2}(P) - 2 \pi s \, \MinTen{k}{r}{s}(P) + \sum_{G \in \cF_{k + 1}(P)} Q(G) \Upsilon_{r - 1}(G) \Theta_{s - 1}(P, G),
    \end{align*}
    where the final summation with respect to $G$ equals $\TenCM{k + 1}{r - 1}{s - 1}{1}(P, \R^{n})$.
    Hence, we obtain recursively
    \begin{align}
      \tfrac{k}{2\pi} \TenCM{k}{r}{s - 2}{1}(P, \R^{n}) & = Q \MinTen{k}{r}{s - 2}(P) + Q \MinTen{k + 1}{r - 1}{s - 1}(P) - 2 \pi s \, \MinTen{k}{r}{s}(P) - 2 \pi (s + 1) \, \MinTen{k + 1}{r - 1}{s + 1}(P) \nonumber \\
      & \qquad + \sum_{G \in \cF_{k + 1}(P)} Q(G) \Upsilon_{r - 2}(G) \Theta_{s}(P, G) \nonumber \\
      & = \sum_{p = 0}^{r} \bigg( Q \MinTen{k + p}{r - p}{s + p - 2}(P) - 2 \pi (s + p) \, \MinTen{k + p}{r - p}{s + p}(P) \bigg), \label{Form_McM_MT_gtCM}
    \end{align}
    where the sum with respect to $p$ effectively goes up to $\min\{ r, n - k \}$.
    This representation has also been derived in \cite[Proposition 2.1]{SvanJens16}
    (more generally, for sets with positive reach). Note that these relations extend to general convex bodies (and then also to polyconvex sets), since
    all functionals involved can be continuously extended to general convex bodies (and are additive) as shown in
    \cite{HugSchneider14, HugSchneider17}.

  \section{The proofs} \label{Sec_Proofs}

    \subsection{The Proofs of the Kinematic Formulae}

      The proofs of the kinematic formulae are applications of the kinematic formulae for tensorial curvature measures (obtained in \cite{HugWeis16c}). Although we could prove Theorem \ref{Thm_KF_MT} directly without  first establishing  Theorem \ref{Thm_KF_tiMT}, we start with the latter in order to emphasize the difference in the coefficients of the appearing translation invariant and the general Minkowski tensors. Since this allows us to simply recall part of the argument, it does not lead to undue repetitions in the proof of Theorem \ref{Thm_KF_MT}.

      \begin{proof}[Proof of Theorem \ref{Thm_KF_tiMT}]
        It is sufficient to prove the assertion for polytopes $P, P' \in \cP^{n}$. The general case follows by a straightforward approximation argument. We denote the integral under investigation by $I$.
        Then Theorem \ref{Thm_KF_tCM} (see \cite[Theorem 4]{HugWeis16b}) with $\beta = \beta' = \R^{n}$ yields
        \begin{align*}
          I & = \sum_{k = j + 1}^{n - 1} \bigg( \sum_{m = 0}^{\lfloor \frac s 2 \rfloor} c_{n, j, k}^{s, 0, m} \, Q^{m} \MinTen{k}{0}{s - 2m} (P) + \sum_{m = 1}^{\lfloor \frac s 2 \rfloor} c_{n, j, k}^{s, 1, m} \, Q^{m - 1} \TenCM{k}{0}{s - 2m}{1} (P, \R^{n}) \bigg) \IVol{n - k + j}{0}{0}(P') \\
          & \qquad + \MinTen{j}{0}{s} (P) \, \IVol{n}{0}{0}(P') + e_{n, j}^{s} \, Q^{\frac s 2} \Phi_n^0(P) \, \IVol{j}{0}{0}(P'),
        \end{align*}
        since $c_{n, j, k}^{s, 1, 0}=0$.
        We conclude from \eqref{Form_McM_MT_gtCM} that
        \begin{align*}
          \TenCM{k}{0}{s - 2m}{1} (K, \R^{n}) = \frac{2 \pi}{k} \, Q \MinTen{k}{0}{s - 2m} (P) - \frac{4 \pi^{2}}{k} (s - 2m + 2) \, \MinTen{k}{0}{s - 2m + 2} (P).
        \end{align*}
        Hence, we get
        \begin{align*}
          I & = \sum_{k = j + 1}^{n - 1} \bigg( \sum_{m = 0}^{\lfloor \frac s 2 \rfloor} c_{n, j, k}^{s, 0, m} \, Q^{m} \MinTen{k}{0}{s - 2m} (P) + \sum_{m = 1}^{\lfloor \frac s 2 \rfloor} \frac{2 \pi}{k} c_{n, j, k}^{s, 1, m} \, Q^{m} \MinTen{k}{0}{s - 2m} (P) \\
          & \qquad - \sum_{m = 0}^{\lfloor \frac s 2 \rfloor - 1} \frac{4 \pi^{2}}{k} (s - 2m) c_{n, j, k}^{s, 1, m + 1} \, Q^{m} \MinTen{k}{0}{s - 2m} (P) \bigg) \IVol{n - k + j}{0}{0}(P') \\
          & \qquad + \MinTen{j}{0}{s} (P) \, \IVol{n}{0}{0}(P') + e_{n, j}^{s}  \, Q^{\frac s 2} \, \IVol{n}{0}{0}(P) \, \IVol{j}{0}{0}(P').
        \end{align*}
        Combining all the sums with respect to $m$ gives
        \begin{align*}
          I & = \sum_{k = j + 1}^{n - 1} \sum_{m = 0}^{\lfloor \frac s 2 \rfloor} \left( c_{n, j, k}^{s, 0, m} + \frac{2 \pi}{k} c_{n, j, k}^{s, 1, m} - \frac{4 \pi^{2}}{k} (s - 2m) c_{n, j, k}^{s, 1, m + 1} \right) Q^{m} \MinTen{k}{0}{s - 2m} (P) \, \IVol{n - k + j}{0}{0} (P') \\
          & \qquad + \MinTen{j}{0}{s} (P) \, \IVol{n}{0}{0}(P') + e_{n, j}^{s}  Q^{\frac s 2} \, \IVol{n}{0}{0}(P) \, \IVol{j}{0}{0}(P'),
        \end{align*}
        which holds as $c_{n, j, k}^{s, 1, 0} = 0$ and, for $m = \lfloor \frac s 2 \rfloor$, either $(s - 2m) = 0$ (if $s$ is even) or $\MinTen{k}{0}{s - 2m} \equiv 0$ (if $s$ is odd).
        Now we simplify the occurring coefficients
        \begin{align*}
           e_{n, j, k}^{s, m, 0} & := c_{n, j, k}^{s, 0, m} + \frac{2 \pi}{k} c_{n, j, k}^{s, 1, m} - \frac{4 \pi^{2}}{k} (s - 2m) c_{n, j, k}^{s, 1, m + 1} \\
          & \phantom{:} = \frac{1}{(4\pi)^{m} m!} \frac {\Gamma(\frac {k} 2 + 1)} {\Gamma(\frac {k + s} 2 + 1)} \alpha_{n, j, k} \bigg( \left( 1 + \frac{2m}{k} \right) \frac {\Gamma(\frac {j + s} 2 - m + 1)} {\Gamma(\frac {j} 2 + 1)} \frac {\Gamma(\frac {k - j} 2 + m)} {\Gamma(\frac{k - j}{2})} \\
          & \qquad \qquad \qquad \qquad \qquad \qquad \qquad - \frac{s - 2m}{k} \frac {\Gamma(\frac {j + s} 2 - m)} {\Gamma(\frac {j} 2 + 1)} \frac {\Gamma(\frac {k - j} 2 + m + 1)} {\Gamma(\frac{k - j}{2})} \bigg) \\
          & \phantom{:} = \frac{1}{(4\pi)^{m} m!} \frac {\Gamma(\frac {k} 2)} {\Gamma(\frac {j} 2)} \frac {\Gamma(\frac {j + s} 2 - m)} {\Gamma(\frac {k + s} 2)} \frac {\Gamma(\frac {k - j} 2 + m)} {\Gamma(\frac{k - j}{2})} \alpha_{n, j, k},
        \end{align*}
        for $0\le j < k < n$ and $m = 0, \ldots, \lfloor \frac s 2 \rfloor - 1$. Altough it is irrelevant here, as explained before, we define the coefficients for $m = \lfloor \frac s 2 \rfloor$ in a slightly different way by
        \begin{align*}
          e_{n, j, k}^{s, \lfloor \frac s 2 \rfloor, 0} & := c_{n, j, k}^{s, 0, \lfloor \frac s 2 \rfloor} + \frac{2 \pi}{k} c_{n, j, k}^{s, 1, \lfloor \frac s 2 \rfloor} \\
          & \phantom{:} = \frac{\frac{k} 2 + \lfloor \frac s 2 \rfloor}{(4\pi)^{\lfloor \frac s 2 \rfloor} \lfloor \frac s 2 \rfloor!} \frac {\Gamma(\frac {k} 2)} {\Gamma(\frac {j} 2 + 1)} \frac {\Gamma(\frac {j + s} 2 - \lfloor \frac s 2 \rfloor + 1)} {\Gamma(\frac {k + s} 2 + 1)} \frac {\Gamma(\frac {k - j} 2 + \lfloor \frac s 2 \rfloor)} {\Gamma(\frac{k - j}{2})} \alpha_{n, j, k},
        \end{align*}
        which turns out to be appropriate in view of Theorem \ref{Thm_KF_MT}. For even $s$ this coincides with the general definition, whereas for odd $s$ these
        two expressions are different.

        The continuation of the coefficients to $k = j\ge 1$ is given by $e_{n, j, j}^{s, m, 0} = \1 \{ m = 0 \}$ (if $j=k=0$ this is true by definition), and to $k=n$ and for even $s$ it is given by
        $$
         e_{n,j , n}^{s, \lfloor\frac{s}{2} \rfloor, 0}=e_{n, j}^{s}.
        $$
        Hence, we can briefly write
        \begin{align*}
          I & = \sum_{k = j}^{n} \sum_{m = 0}^{\lfloor \frac s 2 \rfloor} e_{n, j, k}^{s, m, 0} \, Q^{m} \MinTen{k}{0}{s - 2m} (P) \, \IVol{n - k + j}{0}{0} (P'),
        \end{align*}
        since $\MinTen{n}{0}{s - 2m}$ vanishes if $m \neq \frac s 2$.
      \end{proof}

      In the proof of the general case, we observe that the coefficients of the translation invariant Minkowski tensors are the same as the ones which we derived in Theorem \ref{Thm_KF_tiMT}. However, the coefficients of the other Minkowski tensors have to be defined in a slightly different way.

      \begin{proof}[Proof of Theorem \ref{Thm_KF_MT}]
        Again, it is sufficient to prove the assertion for polytopes $P, P' \in \cP^{n}$. The general case follows by an approximation argument. We denote the integral under investigation by~$I$.
        Then Theorem \ref{Thm_KF_tCM} with $\beta =  \R^{n}$ and $\beta' = \R^{n}$ yields, as in the proof of Theorem \ref{Thm_KF_tiMT},
        \begin{align*}
          I & = \sum_{k = j + 1}^{n - 1} \bigg( \sum_{m = 0}^{\lfloor \frac s 2 \rfloor} c_{n, j, k}^{s, 0, m} \, Q^{m} \MinTen{k}{r}{s - 2m} (P) + \sum_{m = 1}^{\lfloor \frac s 2 \rfloor} c_{n, j, k}^{s, 1, m} \, Q^{m - 1} \TenCM{k}{r}{s - 2m}{1} (P, \R^{n}) \bigg) \IVol{n - k + j}{0}{0}(P') \\
          & \qquad + \MinTen{j}{r}{s} (P) \, \IVol{n}{0}{0}(P') + e_{n, j}^{s}  \, Q^{\frac s 2} \MinTen{n}{r}{0}(P) \, \IVol{j}{0}{0}(P').
        \end{align*}
        We conclude from \eqref{Form_McM_MT_gtCM}
        \begin{align*}
          I & = \sum_{k = j + 1}^{n - 1} \sum_{m = 0}^{\lfloor \frac s 2 \rfloor} c_{n, j, k}^{s, 0, m} \, Q^{m} \MinTen{k}{r}{s - 2m} (P) \, \IVol{n - k + j}{0}{0}(P') \\
          & \qquad + \sum_{k = j + 1}^{n - 1} \frac{2\pi}{k} \sum_{p = 0}^{r} \bigg( \sum_{m = 1}^{\lfloor \frac s 2 \rfloor} c_{n, j, k}^{s, 1, m} \, Q^{m} \MinTen{k + p}{r - p}{s - 2m + p}(P) \\
          & \qquad \qquad \qquad - 2\pi \sum_{m = 1}^{\lfloor \frac s 2 \rfloor} (s - 2m + p + 2) c_{n, j, k}^{s, 1, m} \, Q^{m - 1} \MinTen{k + p}{r - p}{s - 2m + p + 2}(P) \bigg) \IVol{n - k + j}{0}{0}(P') \\
          & \qquad + \MinTen{j}{r}{s} (P) \, \IVol{n}{0}{0}(P') + e_{n, j}^{s} \, Q^{\frac s 2} \MinTen{n}{r}{0}(P) \, \IVol{j}{0}{0}(P').
        \end{align*}
        For $p = 0$, we can combine all of the coefficients as in the proof of the translation invariant case in Theorem \ref{Thm_KF_tiMT}. For this purpose it is important that the definition of the coefficients for $m = \frac{s - 1} 2$ (if $s$ is odd) differs from the general ones, as for $r > 0$, we have $\MinTen{k}{r}{1} \not \equiv 0$ in general. Thus, we obtain
        \begin{align*}
          I & = \sum_{k = j}^{n} \sum_{m = 0}^{\lfloor \frac s 2 \rfloor} e_{n, j, k}^{s, m, 0} \, Q^{m} \MinTen{k}{r}{s - 2m} (P) \, \IVol{n - k + j}{0}{0}(P') \\
          & \quad + \sum_{k = j + 1}^{n - 1} \frac{2\pi}{k} \sum_{p = 1}^{r} \bigg( \sum_{m = 0}^{\lfloor \frac s 2 \rfloor - 1} \big( c_{n, j, k}^{s, 1, m}  - 2\pi (s - 2m + p) c_{n, j, k}^{s, 1, m + 1} \big) \, Q^{m} \MinTen{k + p}{r - p}{s - 2m + p}(P) \\
          & \qquad \qquad \qquad \qquad \qquad + c_{n, j, k}^{s, 1, \lfloor \frac s 2 \rfloor} \, Q^{\lfloor \frac s 2 \rfloor} \MinTen{k + p}{r - p}{s - 2\lfloor \frac s 2 \rfloor + p}(P) \bigg) \IVol{n - k + j}{0}{0}(P'),
        \end{align*}
        as $c_{n, j, k}^{s, 1, 0} = 0$.
        Now we rename the remaining coefficients as
        \begin{align*}
          e_{n, j, k}^{s, m, p} & := \tfrac{2\pi}{k} \left( c_{n, j, k}^{s, 1, m}  - 2\pi (s - 2m + p) c_{n, j, k}^{s, 1, m + 1} \right),
        \end{align*}
        which we simplify via
        \begin{align*}
          e_{n, j, k}^{s, m, p} & = \frac{1}{(4\pi)^{m} m!} \frac {\Gamma(\frac {k} 2 + 1)} {\Gamma(\frac {j} 2 + 1)} \frac {\Gamma(\frac {j + s} 2 - m)} {\Gamma(\frac {k + s} 2 + 1)} \frac {\Gamma(\frac {k - j} 2 + m)} {\Gamma(\frac{k - j}{2})} \alpha_{n, j, k} \\
          & \qquad \qquad \times \underbrace{\tfrac{2}{k} \left( m \left( \tfrac{j + s}{2} - m \right) - \tfrac{s - 2m + p}{2} \left( \tfrac{k - j}{2} + m \right) \right)}_{ = m \frac{k - p}{k} - \frac{s + p}{2} \frac{k - j}{k}},
        \end{align*}
        and denote $e_{n, j, k}^{s, \lfloor \frac s 2 \rfloor, p} := \frac{2\pi}{k} c_{n, j, k}^{s, 1, \lfloor \frac s 2 \rfloor}$, for $p > 0$, which gives the assertion.
      \end{proof}

    \subsection{The Proofs of the Crofton Formulae}

      The connection of the Crofton formula and the kinematic formula (for the scalar (local) case, see the proof of Theorem~4.4.5 in  \cite{Schneider14}, for the tensorial (local) case, see the proofs of Theorem~1 and Theorem~4 in \cite{HugWeis16c}) cannot be used to derive the Crofton formulae for Minkowski tensors, as this requires local arguments which are only available for the measure-valued valuations.
      However, we can prove the Crofton formulae by globalizing the corresponding Crofton formulae for the global curvature measures and by a subsequent application of Lemma \ref{Lem_McM_MT_gtCM} to these results.
      Again, we treat the translation invariant and the general case separately, and start with the former.

      \begin{proof}[Proof of Theorem \ref{Thm_CF_tiMT_j<k}]
        We establish the assertion for a polytope $P \in \cP^{n}$.  Let us denote the integral under investigation by $I$.
        Then Theorem~\ref{Thm_CF_tCM_j<k} (see \cite[Theorem~4]{HugWeis16c}) with $\beta = \R^{n}$ and $r=0$ yields
        \begin{align*}
          I & = \sum_{m = 0}^{\lfloor \frac s 2 \rfloor} c_{n, j, n - k + j}^{s, 0, m} \, Q^{m} \MinTen{n - k + j}{0}{s - 2m}(P) + \sum_{m = 1}^{\lfloor \frac s 2 \rfloor} c_{n, j, n - k + j}^{s, 1, m} \, Q^{m - 1} \TenCM{n - k + j}{0}{s - 2m}{1} (P, \R^{n}).
        \end{align*}
        As in the preceding proofs, we obtain from \eqref{Form_McM_MT_gtCM}
        \begin{align*}
          I & = \sum_{m = 0}^{\lfloor \frac s 2 \rfloor} c_{n, j, n - k + j}^{s, 0, m} \, Q^{m} \MinTen{n - k + j}{0}{s - 2m}(P) + \tfrac{2 \pi}{n - k + j} \sum_{m = 1}^{\lfloor \frac s 2 \rfloor} c_{n, j, n - k + j}^{s, 1, m} \, Q^{m} \MinTen{n - k + j}{0}{s - 2m} (P) \\
          & \qquad - \tfrac{4 \pi^{2}}{n - k + j} \sum_{m = 1}^{\lfloor \frac s 2 \rfloor} (s - 2m + 2) c_{n, j, n - k + j}^{s, 1, m} \, Q^{m - 1} \MinTen{n - k + j}{0}{s - 2m + 2} (P).
        \end{align*}
        Comparing this to the proof of Theorem \ref{Thm_KF_tiMT}, we observe, that we simply obtain the same coefficients in a different order. Thus, the proof is already complete.
      \end{proof}

      In the proof of the Crofton formulae for general Minkowski tensors, we observe that the coefficients of the translation invariant Minkowski tensors are the same as the ones which we derived in Theorem \ref{Thm_KF_tiMT}. However, the coefficients of the other Minkowski tensors have to be defined in a slightly different way.

      \begin{proof}[Proof of Theorem \ref{Thm_CF_MT_j<k}]
      We only prove the assertion for a polytope $P \in \cP^{n}$ and refer to an approximation argument for the case of general convex bodies. We denote the integral under investigation by $I$.
      Then Theorem~\ref{Thm_CF_tCM_j<k} (see \cite[Theorem~4]{HugWeis16c}) with $\beta = \R^{n}$ yields
      \begin{align*}
        I & = \sum_{m = 0}^{\lfloor \frac s 2 \rfloor} c_{n, j, n - k + j}^{s, 0, m} \, Q^{m} \MinTen{n - k + j}{r}{s - 2m}(P) + \sum_{m = 1}^{\lfloor \frac s 2 \rfloor} c_{n, j, n - k + j}^{s, 1, m} \, Q^{m - 1} \TenCM{n - k + j}{r}{s - 2m}{1} (P, \R^{n}).
      \end{align*}
      As before, we conclude from \eqref{Form_McM_MT_gtCM}
      \begin{align*}
        I & = \sum_{m = 0}^{\lfloor \frac s 2 \rfloor} c_{n, j, n - k + j}^{s, 0, m} \, Q^{m} \MinTen{n - k + j}{r}{s - 2m}(P) + \tfrac{2\pi}{n - k + j} \sum_{m = 1}^{\lfloor \frac s 2 \rfloor} \sum_{p = 0}^{r} c_{n, j, n - k + j}^{s, 1, m} \, Q^{m} \MinTen{n - k + j + p}{r - p}{s - 2m + p}(P) \\
        & \qquad - \tfrac{4\pi^{2}}{n - k + j} \sum_{m = 1}^{\lfloor \frac s 2 \rfloor} \sum_{p = 0}^{r} c_{n, j, n - k + j}^{s, 1, m} (s - 2m + p + 2) \, Q^{m - 1}\MinTen{n - k + j + p}{r - p}{s - 2m + p + 2}(P).
      \end{align*}
      Similarly to the preceding proof, we observe that we obtain the same coefficients as in the proof of Theorem \ref{Thm_KF_MT}, which concludes the argument.
    \end{proof}

  \appendix
  \numberwithin{equation}{section}
  \numberwithin{Satz}{section}

  \section{Integral formulae for tensorial curvature measures} \label{Sec_Appendix}

    In this section, we give a brief definition of the (generalized) tensorial curvature measures, which are applied in this work. Furthermore, we state the kinematic and Crofton formulae which involve the tensorial curvature measures (see~ \cite{HugWeis16b, HugWeis16c}). These form the basis of the proofs of the corresponding integral formulae for the Minkowski tensors.

  \subsection{The (generalized) tensorial curvature measures}

    First, we recall the definitions of the (generalized) tensorial curvature measures. The \textit{tensorial curvature measures} are tensor-valued generalizations of the curvature measures  and thus local versions of the Minkowski tensors (which in turn are tensor-valued generalizations of the intrinsic volumes). For $K \in \cK^{n}$, they are defined  as the Borel measures $ \TenCM{j}{r}{s}{0} (K, \cdot)$, $r,s\in\N_0$,  on $\mathcal{B}(\R^n)$ which are given by
    \begin{equation*}
      \TenCM{j}{r}{s}{0} (K, \beta):=  \frac{1}{r! s!} \frac{\omega_{n - j}}{\omega_{n - j + s}} \int _{\beta \times \Sn} x^r u^s \, \Lambda_j(K, \intd (x, u)),
    \end{equation*}
    for $j\in\{0,\ldots,n-1\}$ and $\beta \in \cB(\R^{n})$. In addition, we define
    $$
    \phi_n^{r,0}(K,\beta):=\phi_n^{r}(K,\beta):=\frac{1}{r!}\int_\beta x^r\, \mathcal{H}^n(\intd x).
    $$
    Hence, $\Phi_n^{r}(K)=\phi_n^r(K,\R^n)$ and, in particular, $\Phi_n^{0}(K)=\phi_n^0(K,\R^n)=V_n(K)$. Also note that
    $\phi_j^{0,0,0}(K,\cdot)=\phi_j(K,\cdot)$ is the $j$th curvature measure of $K$ and $\phi_j(K,\R^n)=V_j(K)$.

    Obviously, the total tensorial curvature measures are just the Minkowski tensors. A modification of the representation \eqref{Form_MT_Pol} can be used to define the \emph{generalized tensorial curvature measure}
		$$
		\TenCM{j}{r}{s}{1} (P, \cdot) ,\qquad j \in\{0, \ldots, n - 1\}, \, r, s \in \N_{0},
		$$
    of a polytope $P\in\cP^n$ as the Borel measure on $\mathcal{B}(\R^n)$ which is given by
    \begin{equation*}
      \TenCM{j}{r}{s}{1} (P, \beta) :=\frac{2\pi}{j} \frac{1}{r! s!} \frac{1}{\omega_{n - j + s}}
			\sum_{F \in \cF_{j}(P)} Q(F) \int _{F \cap \beta} x^r \, \cH^{j}(\intd x)
			\int _{N(P,F) \cap \Sn} u^{s} \, \cH^{n - j - 1} (\intd u),
    \end{equation*}
    for $\beta \in \cB(\R^{n})$. For the just defined generalized tensorial curvature measures, there also exist continuous extensions to $\cK^{n}$ (as explained in \cite[Section 2]{HugWeis16b} and according to \cite{HugSchneider14,HugSchneider17}).
    For more details on the (generalized) tensorial curvature measures see \cite[Section~2]{HugWeis16b}.

  \subsection{Integral formulae for the tensorial curvature measures}

    Next, we state the integral formulae, which we need in the proofs in Section \ref{Sec_Proofs}. We start with the kinematic formula for tensorial curvature measures (see \cite[Theorem 4]{HugWeis16b}).

    \begin{Theorem} \label{Thm_KF_tCM}
      For $K, K' \in \cK^n$, $\beta, \beta' \in \cB(\R^n)$ and $j, r, s \in \N_0$ with $j \leq n$,
      \begin{align*}
        &\int_{\GOp_n} \TenCM{j}{r}{s}{0} (K \cap g K', \beta \cap g \beta') \, \mu( \intd g) \\
        & \qquad = \sum_{k = j + 1}^{n - 1} \sum_{m = 0}^{\lfloor \frac s 2 \rfloor} \sum_{i = 0}^{1} c_{n, j, k}^{s, i, m} \, Q^{m - i} \TenCM{k}{r}{s - 2m}{i} (K, \beta) \phi_{n-k+j}(K', \beta') \\
        & \qquad \qquad + \TenCM{j}{r}{s}{0} (K, \beta) \phi_n(K', \beta')
          + e_{n, j}^{s} \, Q^{ \frac s 2 }\phi_n^r (K, \beta) \phi_j (K', \beta'),
      \end{align*}
      where
      \begin{align*}
        c_{n, j, k}^{s, i, m} : = \frac{1}{(4\pi)^{m} m!} \frac{\binom{m}{i}}{\pi^{i}} \frac {\Gamma(\frac {k} 2 + 1)} {\Gamma(\frac {j} 2 + 1)} \frac {\Gamma(\frac {j + s} 2 - m + 1)} {\Gamma(\frac {k+s} 2 + 1)} \frac {\Gamma(\frac {k - j} 2 + m)} {\Gamma(\frac{k - j}{2})} \alpha_{n, j, k}
      \end{align*}
      with $\alpha_{n, j, k}$ as in \eqref{Form_Princ_KF} and
      \begin{align*}
        e_{n, j}^{s} & := \1 \{ s \text{ even} \} \frac{1}{(2\pi)^{s} \frac{s}{2}!} \frac{\Gamma(\frac{n - j + s}{2})}{\Gamma(\frac{n - j}{2})} \frac{\omega_{n + s}}{\omega_{n}}.
      \end{align*}
    \end{Theorem}

    Stating the Crofton formulae, we distinguish the cases $k = j$ and $j < k$ (see \cite[Theorem~1 and Theorem~4]{HugWeis16c}).

    \begin{Theorem} \label{Thm_CF_tCM_j=k}
      Let $P \in \cP^n$, $\beta \in \cB(\R^n)$ and $k, r, s, l \in \N_{0}$ with $k \leq n$. Then,
      \begin{align*}
        \int_{\AOp(n, k)} \TenCM{k}{r}{s}{0} (P \cap E, \beta \cap E) \, \mu_k( \intd E) = e_{n, j}^{s} \, Q^{\frac s 2}\phi_n^r(P, \beta),
      \end{align*}
      where $e_{n, j}^{s}$ is defined as in Theorem \ref{Thm_KF_tCM}.
    \end{Theorem}

    \begin{Theorem} \label{Thm_CF_tCM_j<k}
      Let $K \in \cK^n$, $\beta \in \cB(\R^n)$ and $j, k, r, s \in \N_{0}$ with $j < k \leq n$. Then,
      \begin{align*}
        \int_{\AOp(n, k)} \TenCM{j}{r}{s}{0} (K \cap E, \beta \cap E) \, \mu_k( \intd E) = \sum_{m = 0}^{\lfloor \frac s 2 \rfloor} \sum_{i = 0}^{1} c_{n, j, n - k + j}^{s, i, m} \, Q^{m - i} \TenCM{n - k + j}{r}{s - 2m}{i} (K, \beta),
      \end{align*}
      where the coefficients $c_{n, j, n - k + j}^{s, i, m}$ are defined as in Theorem \ref{Thm_KF_tCM} and $c_{n, j, j}^{s, i, m} = \1\{ m = i = 0 \}$.
    \end{Theorem}

\end{document}